\documentclass [a4paper, 12pt, reqno]{amsart}
\usepackage [latin1]{inputenc}
\usepackage{amscd}
\usepackage{epsfig}
\usepackage{amssymb}
\usepackage{amsmath}
\usepackage{amscd}
\usepackage{mathrsfs}\usepackage{mathrsfs}
\usepackage{amsthm}
\usepackage[T1]{fontenc}
\usepackage{ae,aecompl}
\usepackage[arrow, matrix, curve]{xy}
\usepackage[breaklinks=true, colorlinks=false]{hyperref}
\usepackage{microtype}
\usepackage{geometry, setspace}
\usepackage{enumitem}

\newlength{\itemLM} 
\setlist{leftmargin=\itemLM}
\SetLabelAlign{right}{\hspace{+2\parindent}\makebox[0\itemLM]{#1}}

\numberwithin{equation}{section}

 \newcommand{\mpar}{\par\medskip} \newcommand{\bpar}{\par\bigskip}
\newtheorem{lem}[equation]{Lemma}
\newcommand{\newterm}{\textit}
\newtheorem{thm}[equation]{Theorem}
\newtheorem{cor}[equation]{Corollary}

\newenvironment{proofX}[1]{\par\noindent\textit{Proof #1.}}{\hfill $\Box$\bpar}

\newcommand{\FormelQed}{\vskip-\belowdisplayskip \vskip-\baselineskip}

\newcommand{\nbd}{\nobreakdash-}

\newcommand{\cC}{\mathcal{C}}   \newcommand{\cF}{\mathcal{F}}   \newcommand{\cL}{\mathcal{L}} \newcommand{\cM}{\mathcal{M}} \newcommand{\cO}{\mathcal{O}}  \newcommand{\cR}{\mathcal{R}}

\newcommand{\am}{\mathfrak{m}}     

\newcommand{\sC}{\mathscr{C}} \newcommand{\sD}{\mathscr{D}}     \newcommand{\sH}{\mathscr{H}} \newcommand{\sJ}{\mathscr{J}}       

  \newcommand{\CC}{\mathbb{C}} \newcommand{\CP}{\mathbb{CP}}  \newcommand{\NN}{\mathbb{N}} \newcommand{\PP}{\mathbb{P}}  \newcommand{\RR}{\mathbb{R}} \newcommand{\ZZ}{\mathbb{Z}}

\newcommand{\ii}{\mathrm{i}}  
 \newcommand{\cpt}{\mathrm{cpt}}   \newcommand{\dom}[1][]{\mathrm{dom}_{#1}\,}     \newcommand{\loc}{\mathrm{loc}} \newcommand{\pr}{\mathrm{pr}}         

    \DeclareMathOperator{\mult}{mult} \DeclareMathOperator{\ord}{ord}     \DeclareMathOperator{\Reg}{Reg} \DeclareMathOperator{\Sing}{Sing}      \DeclareMathOperator{\supp}{supp}    \DeclareMathOperator{\qhodge}{\quer\ast}

\newcommand{\wcO}{\widehat\cO}

\newcommand{\quer}{\overline} \newcommand{\dbar}{\overline{\partial}}  \newcommand{\dq}{\overline{\partial}}

\newcommand{\cf}{cf.\ } \newcommand{\eg}{e.\,g.\ } \newcommand{\ie}{i.\,e., } 

\newcommand{\aus}{\subset}  \newcommand{\um}{\supset}  \newcommand{\minus}{\setminus} \newcommand{\leer}{\varnothing} \newcommand{\durch}{\cap} \newcommand{\ver}{\cup}     \newcommand{\kreuz}{\times}     \newcommand{\tensor}{\otimes} \newcommand{\genau}{\Leftrightarrow}

\newcommand{\abb}{\rightarrow} \newcommand{\auf}{\mapsto} \newcommand{\iso}{\cong} \newcommand{\nach}{\circ}

\newcommand{\kl}{\leq} \newcommand{\gr}{\geq} \newcommand{\ungl}{\neq}

\newcommand{\eps}{\varepsilon} \newcommand{\ph}{\varphi}  \newcommand{\thet}{\vartheta} 

\newcommand{\spfrac}[2]{{}^{#1}\!/\!_{#2}}
\newcommand{\konv}[2]{\hbox{ if }#1\abb #2}
\newcommand{\cOL}{\cO_{L^2}(\Delta)} \newcommand{\cOLm}[1]{t^{#1}\cOL} \newcommand{\cOmL}{\Omega^1_{L^2}(\Delta)} \newcommand{\cOmLm}[1]{t^{#1}\cOmL}

\makeatletter
\def\newrefformat#1#2{%
  \@namedef{pr@#1}##1{#2}}
\newrefformat{eq}{\textup{(\ref{#1})}} \newrefformat{lem}{Lemma~\ref{#1}} \newrefformat{thm}{Theorem~\ref{#1}} \newrefformat{cha}{Chapter~\ref{#1}} \newrefformat{sec}{Section~\ref{#1}}
\newrefformat{tab}{Table~\ref{#1} on page \pageref{#1}} \newrefformat{fig}{Figure~\ref{#1} on page \pageref{#1}}
\def\prettyref#1{\@prettyref#1:}
\def\@prettyref#1:#2:{%
  \expandafter\ifx\csname pr@#1\endcsname\relax%
    \PackageWarning{prettyref}{Reference format #1\space undefined}%
    \ref{#1:#2}%
  \else%
    \csname pr@#1\endcsname{#1:#2}%
  \fi%
}

\def\section{\@startsection{section}{1}%
  \z@{1.2\linespacing\@plus\linespacing}{.7\linespacing}%
  {\large\normalfont\bf\centering}}

\makeatother

\title[$L^2$-Riemann-Roch for singular complex curves]  {$L^2$-Riemann-Roch for singular complex curves}

\author[J. Ruppenthal and M. Sera]{Jean Ruppenthal and Martin Sera}
\address{Department of Mathematics, University of Wuppertal, Gau{\ss}str. 20, 42119 Wuppertal, Germany.}
\email{ruppenthal@uni-wuppertal.de, sera@math.uni-wuppertal.de}

\date{\today}

\hypersetup{
  pdftoolbar=true,
  pdfmenubar=true,
  pdftitle={L\textasciicircum 2-Riemann-Roch for singular complex curves},
  pdfauthor={Jean Ruppenthal and Martin Sera},
  pdfkeywords={Dolbeault operator, L\textasciicircum 2-theory, singular complex curves, Riemann-Roch theorem},
pdfborder= 0 0 .25,
}


\begin{document}
\begin{abstract}
We present a comprehensive $L^2$-theory for the $\dbar$-operator on singular complex curves,
including $L^2$-versions of the Riemann-Roch theorem and some applications.
\end{abstract}

\maketitle

\vspace{-1mm} 
\section{Introduction}

The $L^2$-theory for the $\dbar$-operator is one of the central tools in complex analysis on complex manifolds, but still not very well developed on singular complex spaces. Just recently, considerable progress has been made in understanding the $L^2$-cohomology of singular complex spaces with isolated singularities. Let $X$ be a Hermitian complex space of pure dimension $n$ and with isolated singularities only. For simplicity, we assume that $X$ is compact. Let $H^{p,q}_w(X^*)$ be the $L^2$-Dolbeault cohomology on the level of $(p,q)$-forms with respect to the $\dq$-operator in the sense of distributions, denoted by $\dbar_w$ in the following, computed on $X^*=\Reg X$. Let $\pi: M\rightarrow X$ be a resolution of singularities with snc exceptional divisor, $Z:=\pi^{-1}(\Sing X)$ the unreduced exceptional divisor. Then it has been shown by {\O}vrelid, Vassiliadou \cite{OvrelidVassiliadou13} and the first author \cite{Ruppenthal11,Ruppenthal14} by different approaches that there exists 
an effective divisor $D\geq Z-|Z|$ on $M$ such that there are
natural isomorphisms
	\begin{equation}\label{eq:history}\begin{split}
	H^{n,q}_w(X^*) &\cong H^{n,q}(M)\ ,\\
	H^{0,q}_w(X^*) &\cong \frac{H^q\big(M,\cO(D)\big)}{H^q_{|Z|}\big(M,\cO(D)\big)}
	\end{split}\end{equation}
for all $0\leq q\leq n$. 
Here, $H^q_{|Z|}$ denotes the cohomology with support on $|Z|$.
If $\dim X \leq 2$, then \eqref{eq:history} holds with the divisor $D=Z-|Z|$
and $H^q_{|Z|}\big(M,\cO(Z-|Z|)\big)=\{0\}$,
so that \eqref{eq:history} gives a nice smooth representation of the $L^2$-cohomology groups $H^{0,q}_w(X^*)$.
In case $\dim X>2$, it is conjectured that \eqref{eq:history} holds with $D=Z-|Z|$ (see \cite{Ruppenthal11}).

\medskip
However, the $L^2$-theory for the $\dbar$-operator developed in \cite{OvrelidVassiliadou13} and \cite{Ruppenthal11,Ruppenthal14} applies
only to $\dim X\geq 2$ (for $\dim X=1$, \eqref{eq:history} has been known before, see \cite{Pardon89, PardonStern91}).
The purpose of the present paper is to give a complete $L^2$-theory for the $\dq$-operator on a singular complex curve,
including $L^2$-versions of the Riemann-Roch theorem, and to understand the appearance of the divisor $Z-|Z|$
in the case $\dim X=1$.

\medskip
Let us explain some of our results in detail. Let $X$ be a Hermitian singular complex space\footnote{ A Hermitian complex space $(X,g)$ is a reduced complex space $X$ with a metric $g$ on the regular part such that the following holds: If $x\in X$ is an arbitrary point there exist a neighborhood $U=U(x)$, a holomorphic embedding of $U$ into a domain $G$ in $\CC^N$ and an ordinary smooth Hermitian metric in $G$ whose restriction to $U$ is $g|_U$.} of dimension $1$, \ie a Hermitian complex curve, and $L\rightarrow X$ a Hermitian holomorphic line bundle. Let $\dbar_w:L^{p,q}(X^\ast,L)\abb L^{p,q+1}(X^\ast,L)$ denote the weak extension of the Cauchy-Riemann operator $\dbar:\sD^{p,q}(X^\ast, L)\abb\sD^{p,q+1}(X^\ast, L)$, \ie the $\dbar$-operator in the sense of distributions. Here, $\sD^{p,q}(X^\ast, L):=\sC^\infty_{\cpt}(X^\ast, \Lambda^{p,q}T^\ast X^\ast\tensor L)$ denotes the set of smooth differential forms with compact support in $X^\ast$ and values in $L$ and $L^{p,q}(X^\ast,L)$ is the set of square-integrable forms with values in $L$ and respect to the Hermitian metrics on $X^\ast$ and $L$.

Let $H_{w}^{p,q}(X^\ast, L)$ denote the $L^2$-Dolbeault cohomology on $X^*$ with respect to $\dbar_w$
and $h_w^{p,q}(X^*,L):= \dim H_{w}^{p,q}(X^\ast, L)$.
Note that the genus $g=g(X)$ of $X$ and the degree $\deg(L)$ of $L$ are well-defined,
even in the presence of singularities (see Section \ref{sec:resolution}).
For a singular point $x\in\Sing X$, we define its modified multiplicity $\mult_x' X$ as follows:
Let $X_j$, $j=1, ..., m$, be the irreducible components of $X$ in the singular point $x$.
Then 
$$\mult_x' X:= \sum_{j=1}^m \left(\mult_x X_j -1\right).$$
Note that regular irreducible components do not contribute to $\mult'_x X$.
In \prettyref{sec:resolution}, we recall the definition of  the multiplicity $\mult_x X_j$ and  present different ways to compute it.

\begin{thm}[$\dbar_w$-Riemann-Roch]\label{thm:RR-w}
Let $X$ be a compact Hermitian complex curve with $m$ irreducible components and $L\abb X$ a holomorphic line bundle. Then
\begin{equation}\label{eq:RR-w1}
h_w^{0,0}(X^\ast,L)-h_w^{0,1}(X^\ast,L)=m-g+\deg(L) + \sum_{x\in\Sing X} \mult_x' X,
\end{equation}
and
\[h_w^{1,1}(X^\ast,L)-h_w^{1,0}(X^\ast,L)=m-g-\deg(L).\]
\end{thm}

Theorem \ref{thm:RR-w} is a corollary of \prettyref{thm:computedCo} which we prove in Section \ref{sec:main}. We also consider an $L^2$-dual version there, \ie an $L^2$-Riemann-Roch theorem for the minimal closed $L^2$-extension of the $\dbar$-operator which we denote by $\dbar_s$ (see Section \ref{ssec:dbar}).

On singular complex curves, the $\dbar_s$-operator is of particular importance because of its relation to weakly holomorphic functions.
Namely, the weakly holomorphic functions are precisely the $\dbar_s$-holomorphic $L^2_\loc$-functions
(for a localized version of the $\dbar_s$-operator, see Section \ref{sec:weakly}).
Let $H_{s,\loc}^{p,q}(X^\ast)$ denote the $L^2_\loc$-Dolbeault cohomology on $X^*$ with respect to $\dbar_s$,
and $\widehat\cO_X$ the sheaf of germs of weakly holomorphic functions on $X$. Then:

\begin{thm}\label{thm:main2}
Let $X$ be a Hermitian complex curve. Then
\begin{eqnarray*}
H^0(X,\widehat\cO_X) &=& H^{0,0}_{s,\loc}(X^*),\\
H^1(X,\widehat\cO_X) &\cong& H^{0,1}_{s,\loc}(X^*).
\end{eqnarray*}
\end{thm}

If $X$ is irreducible and compact, then $\dim H^0(X,\widehat\cO_X)=1$, $\dim H^1(X,\widehat\cO_X)=g(X)$.
We prove Theorem \ref{thm:main2} in Section \ref{sec:weakly}.

\medskip
To exemplify the use of $L^2$-theory for the $\dbar$-operator on a singular complex space,
in particular the $L^2$-Riemann-Roch theorem, we give in Section \ref{sec:application} two applications.
There, we use our $L^2$-theory to give alternative proofs of two well-known facts.
First, we show that each compact complex curve can be realized as a ramified covering of $\CC\PP^1$.
Second, we show that a positive holomorphic line bundle over a compact complex curve is ample,
yielding that any compact complex curve is projective.

\medskip
Let us clarify the relation to previous work of others.
In the case of complex curves, \eqref{eq:history} was in essence discovered by Pardon \cite{Pardon89},
and one can deduce parts of Theorem \ref{thm:computedCo} and the second statement of Corollary \ref{thm:RR-s} from Pardon's work by 
some additional arguments on the regularity of the $\dbar$-operator.
The first part of Corollary \ref{thm:RR-s} was discovered by Haskell \cite{Haskell89},
and from that one can deduce the second statement of Theorem \ref{thm:RR-w} by use of $L^2$-Serre duality.
Moreover, Theorem \ref{thm:RR-w} was proved in essence by Br\"uning, Peyerimhoff and Schr\"oder in \cite{BrueningPeyerimhoffSchroeder90} and \cite{Schroeder89}
by computing the indices of the $\dbar_w$- and the $\dbar_s$-operator.

The new point in the present work is that we can put all the partial results mentioned above
in the general framework of a comprehensive $L^2$-theory. From that, we draw also a new understanding of weakly holomorphic functions
(Theorem \ref{thm:main2}) and of the divisor $Z-|Z|$. Moreover, all the previous work has been done only for forms
with values in the trivial bundle (except of \cite{Schroeder89}), whereas we incorporate line bundles.
This is essential for applications as we illustrate by the examples mentioned above.

\bigskip

{\bf Acknowledgment.}
{This research was supported by the Deutsche Forschungsgemeinschaft (DFG, German Research Foundation), grant RU 1474/2 within DFG's Emmy Noether Programme. The authors thank Eduardo S. Zeron for interesting and fruitful discussions
and are grateful to the unknown referee for several suggestions which helped to improve the paper.}

\bpar
\section{Preliminaries}\label{sec:preliminaries}

\mpar
\subsection{Closed extensions of the Cauchy-Riemann operator}\label{ssec:dbar}

Let $X$ be a complex curve and $X^\ast:=\Reg X$ the set of regular points. 
We assume that $X$ is a Hermitian complex space in the sense that $X^\ast$ carries a Hermitian metric $\gamma$
which is locally given as the restriction of the metric of the ambient space when $X$ is embedded holomorphically
into some complex number space.

We denote by $\sD^{p,q}(X^\ast)$ the smooth differential forms of degree $(p,q)$ with compact support in $X^\ast$ (test forms) 
and by $L^{p,q}(X^\ast)$ the set of square-integrable forms with respect to the metric $\gamma$ on $X^\ast$.

Let $\dbar_s:L^{p,q}(X^\ast)\abb L^{p,q+1}(X^\ast)$ be the minimal (strong) closed 
$L^2$-extension of the Cauchy-Riemann operator $\dbar:\sD^{p,q}(X^\ast)\abb\sD^{p,q+1}(X^\ast)$, 
\ie $\dbar_s$ is defined by the closure of the graph of $\dbar$ in $L^{p,q}(X^\ast)\times L^{p,q+1}(X^\ast)$. 
$\dbar_w:L^{p,q}(X^\ast)\abb L^{p,q+1}(X^\ast)$ is the maximal (weak) closed $L^2$-extension of $\dbar$,
\ie $\dbar_w$ is defined in sense of distributions. 
We denote by $H_{w/s}^{p,q}(X^\ast)$ the Dolbeault cohomology with respect to $\dbar_w$ or $\dbar_s$, respectively, 
and by $h_{w/s}^{p,q}(X^\ast)$ the dimension of $H_{w/s}^{p,q}(X^\ast)$.

Let $\thet:\sD^{p,q+1}(X^\ast)\abb\sD^{p,q}(X^\ast)$ be the formal adjoint of $\dbar$
and $\thet_{s/w}:=\dbar_{w/s}^\ast$ the Hilbert-space adjoint of $\dbar_{w/s}$.
This notation makes sense as $\thet_{w/s}$ is in fact the maximal (weak) or minimal (strong), respectively, $L^2$-extension of $\thet$. 
Let $\qhodge:L^{p,q}(X^\ast)\abb L^{1-p,1-q}(X^\ast)$ be the conjugated Hodge-$*$-operator with respect to the metric $\gamma$.
Then we have $\thet_{w/s}=-\qhodge\dbar_{w/s}\qhodge$.

Let $L\abb X$ be a Hermitian holomorphic line bundle on $X$ with an (arbitrary) metric on $L$ which is smooth on the whole of $X$. 
We define $\sD^{p,q}(X^\ast, L):=\sC^\infty_{\cpt}(X^\ast,$ $\Lambda^{p,q}T^\ast X^\ast\tensor L)$ as the smooth $(p,q)$-forms with compact support 
and values in $L$, and $L^{p,q}(X^\ast,L)$ as the Hilbert space of square-integrable forms with values in $L$. We consider the Cauchy-Riemann operator $\dbar:\sD^{p,q}(X^\ast, L)\abb\sD^{p,q+1}(X^\ast, L)$ locally given by $\dbar:\sD^{p,q}(X^\ast)\abb\sD^{p,q+1}(X^\ast)$. Since $\dbar$ commutes with the trivializations of the holomorphic line bundle, $\dbar$ is well defined. We get the weak and strong extensions $\dbar_w$, $\dbar_s:L^{p,q}(X^\ast,L)\abb L^{p,q+1}(X^\ast,L)$ and the cohomology $H_{w/s}^{p,q}(X^\ast, L)$ as above.

\medskip
In Section \ref{sec:local}, we will study also the following other closed extensions of $\dbar$ besides the minimal $\dbar_s$ and the maximal $\dbar_w$.
Let $\Omega\Subset\CC^n$ be a domain,
and $X\aus \Omega$ an analytic set of dimension one with $\Sing X=\{0\}$. 
We can interpret $\dbar_s$ as $\dbar_w$ with certain boundary conditions. 
The boundary of $X^\ast$ consists of two parts,
the singular point $\{0\}$ and the boundary at $\partial \Omega$: $\partial X=\partial X^\ast\minus \{0\}$. 
Therefore, there are two boundary conditions. 
Let $\dbar_{s,w}$ denote the closed $L^2$-extension which satisfies the boundary condition at $\{0\}$, 
\ie $f\in\dom\dbar_{s,w}$ iff $f\in\dom\dbar_{w}$ and there is a sequence $\{f_j\}$ in $\dom\dbar_{w}$ such that $\supp f_j\durch\{0\} =\leer$, $f_j\abb f$, and $\dbar_w f_j\abb \dbar_w f$ in $L^2$. $\dbar_{w,s}$ denotes the extension which satisfies the boundary condition at $\partial X$, \ie $f\in\dom\dbar_{w,s}$ iff $f\in\dom\dbar_{w}$ and there is a sequence $\{f_j\}$ in $\dom\dbar_{w}$ such that $\supp f_j\durch \partial X =\leer$, $f_j\abb f$, and $\dbar_w f_j\abb \dbar_w f$ in $L^2$. 
We define the adjoint operators 
$$\thet_{s,w}:=-\qhodge\dbar_{s,w}\qhodge\ \ \mbox{ and }\ \ \thet_{w,s}:=-\qhodge\dbar_{w,s}\qhodge,$$
which we can realize as Hilbert-space adjoint operators (see \cite[Lem.~5.1]{Ruppenthal14}):
 
\begin{lem}\label{lem:comp-adjoint} 
The Hilbert-space adjoints $\dbar_{s,w}^*$ and $\dbar_{w,s}^*$
satisfy the representations $\dbar_{s,w}^\ast=\thet_{w,s}=-\qhodge\dbar_{w,s}\qhodge$
and $\dbar_{w,s}^\ast=\thet_{s,w}=-\qhodge\dbar_{s,w}\qhodge$, respectively.
\end{lem}

\mpar
\subsection{Resolution of complex curves, divisors, line bundles}\label{sec:resolution}

Every (reduced) complex space $X$ (which is countable at infinity) has a resolution of singularities $\pi:M\abb X$, \ie there are a complex manifold $M$,
a proper complex subspace $S$ of $X$ which contains the singular locus of $X$ and a proper holomorphic map $\pi:M\abb X$ such that the restriction $M\minus\pi^{-1}(S)\abb X\minus S$ of $\pi$ is biholomorphic,
and $\pi^{-1}(S)$ is the locally finite union of smooth hypersurfaces (see \cite{Hironaka64} and \cite[Thm.~7.1]{Hironaka77}).

\par
If $X$ is a compact complex curve, then such a resolution is given just by the normalization of the curve,
and it is unique up to biholomorphism: 
Let $\pi_1:M_1\abb X$ and $\pi_2:M_2\abb X$ be two resolutions of $X$. 
Then $\psi:=\pi_2^{-1}\nach\pi_1: M_1\minus\pi_1^{-1}(\Sing X)\abb M_2\minus\pi_2^{-1}(\Sing X)$ is biholomorphic and bounded in the singular locus. 
Yet, $\pi_i^{-1}(\Sing X)$ consist of isolated points. Therefore, $\psi$ has a (bi-) holomorphic extension.

\mpar
Let $\pi:M\abb X$ be a resolution of a compact complex curve $X$. 
We define the \newterm{genus} of $X$ by the genus of the resolution
	\[g(X):=h^1(M)=\dim H^1(M,\cO).\]
If $X$ has more than one irreducible component,
then $M$ is not connected and $h^1(M)$ is the sum of the genera of the connected components.
Since the resolution is unique up to biholomorphism, this is well\nbd defined.

\mpar
Throughout the article (except of \prettyref{sec:cover}), we will work with divisors on compact Riemann surfaces only. 
Therefore, there is no difference between Cartier and Weil divisors, and we can associate to each line bundle a divisor.

\mpar
Let $L\abb X$ be a holomorphic line bundle. 
Then the pull-back $\pi^\ast L\rightarrow M$ is well-defined by the pull-back of the transition functions of the line bundle.
There is a divisor $D$ on $M$ associated to $\pi^\ast L$ such that $\cO(\pi^\ast L)\iso \cO(D)$\footnote{
We denote by $\cO(D)$ the sheaf of germs of holomorphic functions $f$ such that $\mbox{div}(f)+D\geq0$.} and $\deg \pi^\ast L=\deg D$. 
The uniqueness of the resolution (up to biholomorphism) implies the independence of $\deg \pi^\ast L$ from $\pi$, so that
	\[\deg L:=\deg \pi^\ast L\]
is also well\nbd defined.

\mpar
For any divisor $D$ on $M$, there exists a holomorphic line bundle $L_D \rightarrow M$ associated to
$D$ such that $\cO(L_D)\cong \cO(D)$. 
The constant function $f= 1$ induces a meromorphic section $s_D$ of $L_D$
such that $\mbox{div}(s_D)=D$. One can then identify sections in $\cO(D)$
with sections in $\cO(L_D)$ by $g\mapsto g\otimes s_D$,
and we denote the inverse mapping by $s\mapsto s\cdot s_D^{-1}$.
If $Y$ is an effective divisor, then $s_Y$ is a holomorphic section of $L_Y$ and $\cO \subset \cO(Y)$. Hence, there is the natural inclusion $\cO(D)\subset \cO(D+Y)$
which induces the inclusion $\cO(L_{D})\subset \cO(L_{D+Y})$ given by
$s\mapsto (s\cdot s_{D}^{-1})\otimes s_{D+Y}$.
For $U\aus M$, we obtain the inclusion
\begin{eqnarray}\label{eq:L2LocInclusion}
L^{p,q}_{\loc}(U,L_D) \hookrightarrow L^{p,q}_{\loc}(U,L_{D+Y}), s\mapsto (s\cdot s_{D}^{-1})\otimes s_{D+Y}.
\end{eqnarray}
Here, $L^{p,q}_{\loc}(U,L_D)$ denotes the locally square-integrable forms with values in $L_D$. This definition is independent of the chosen Hermitian metric on $L_D$.
If $M$ is compact, all metrics are equivalent and we get the inclusion
\begin{equation}\label{eq:L2inclusion}
L^{p,q}(M, L_D) \subset L^{p,q}(M, L_{D+Y}).
\end{equation}

\mpar
Let $Z:=\pi^{-1}(\Sing X)$ be the unreduced exceptional divisor and $|Z|$ the underlying reduced divisor.
Then $\deg(Z-|Z|)$ is independent of the resolution as well. 
We will discuss some alternative ways to compute $\deg Z$.

\mpar
Locally, the resolution is given by the Puiseux parametrization:
Let $A$ be an analytic set of dimension 1 in $\Omega\Subset\CC^n$ with $\Sing A=\{0\}$ which is irreducible at $0$. 
Shrinking $\Omega$, there are coordinates $z,w_1,...,w_{n-1}$ around $0$ 
such that $A$ is contained in the cone $\|w\|\kl C|z|$, $w=(w_1,..\,,w_{n-1})$. 
The projection $\pr_z:A\abb\CC_z$ on the $z$-coordinate is a finite ramified covering. Let $s$ be the number of the sheets of $\pr_z$.
Generic choice of the coordinates gives the same number of sheets $s$, called the multiplicity $\mult_0 A$ of $A$ in $\{0\}$. 
There exists a parametrization $\pi:\Delta\abb A, t\auf (t^s,w_1(t),...,w_{n-1}(t))$, where $\Delta:=\{t\in\CC:|t|<1\}$; 
\cf \eg \cite[Sect.~6.1]{Chirka89}. $\pi$ is called the Puiseux parametrization.
The unreduced exceptional divisor is just $Z=(\pi^{-1}(z))=(t^s)$, and so $\deg Z=s$.
\par
The number of sheets of the covering $\pr_z$ is also equal to the Lelong number $\nu([A],0)$ of the positive closed current $[A]$ given by the integration over $A$ (see \cite[Prop.~2 in \S~3.15]{Chirka89}, \cite[Thm.~7.7]{DemaillyAG} or \cite[\S~3.2]{GriffithsHarris78}).

\mpar
The tangent cone gives another way to compute $\mult_0A$. For a holomorphic function $f$ on $\Omega$, let $f=\sum_{k=k_0}^\infty f_{k}$ be the decomposition in homogeneous polynomials $f_{k}$ of degree $k$ with $f_{k_0}\ungl 0$ (choosing a smaller $\Omega$) and $f^\ast:=f_{k_0}\ungl 0$ be the \newterm{initial homogeneous polynomial of $f$}.
If $A$ is given by the ideal sheaf $\sJ\!_A$, then
	\[C_0(A)=\left\{\alpha\in\CC^n: f^{\ast}(\alpha)=0\ \forall f\in \sJ\!_{A,0}\right\}\aus T_0\CC^n\]
is called the \newterm{tangent cone} of $A$ in $0$ (\cf \cite[Sect.~8.4]{Chirka89}).
The natural projection $\CC^n\minus 0\abb\CP^{n-1}$ maps $C_0(A)$ on a projective variety $\widetilde{C_0}(A)$. 
The degree $\deg Y$ of a projective variety $Y$ in $\CP^{n-1}$ of dimension $p$ is defined as the class of $Y$ in $H_{2p}(\CP^{n-1},\ZZ)\iso \ZZ$,
and $\mult_0 A=\deg \widetilde{C_0}(A)$ (see Sect.~2 of \cite[\S\,1.3]{GriffithsHarris78}). In the case of an irreducible complex curve $A$, 
note that $\widetilde{C_0}(A)$ is just a point of multiplicity $\mult_0 A$.
\mpar
All in all, we have
	\[\deg Z=\mult_0 A=\nu([A],0)=\deg \widetilde{C_0}(A).\]

\mpar
\subsection{Extension theorems}

We need the following extension theorem. Let $\Delta$ be the unit disc in $\CC$ and $\Delta^*:=\Delta\minus\{0\}$.

\begin{thm}[$L^2$-extension]\label{thm:extension} If $u\in L_\loc^{p,0}(\Delta)$ and $v\in L_\loc^{p,1}(\Delta)$ satisfy $\dbar u=v$ on $\Delta^\ast$ in the sense of distributions, then $\dbar u=v$ on $\Delta$.\end{thm}

A more general statement is true for domains in $\CC^n$ and proper analytic subsets of arbitrary codimension, \cf \eg \cite[Thm.~3.2]{Ruppenthal09IntJMath}.

\mpar
If $A\aus \Omega\Subset\CC^n$ is a pure dimensional analytic set, let $\wcO=\wcO_A$ be the normalization sheaf of $\cO_A$ which is defined stalk-wise by the integral closure of $\cO_{A,x}$ in the sheaf $\cM_{A,x}$ of meromorphic functions for all $x\in A$. A function in $\wcO(U)$, $U\aus A$ open, is called \newterm{weakly holomorphic}. Weakly holomorphic functions are holomorphic in regular points of $A$ and bounded in singular points. If $A$ is locally irreducible, then weakly holomorphic functions are continuous in $\Sing A$ (cf. \eg \cite[\S~VI.4]{GrauertRemmertCAS}.)
\par
The classical Riemann extension theorem generalizes to the following result (see \eg \cite[Sect.~VII.4.1]{GrauertRemmertCAS}).

\begin{thm}[Riemann extension]\label{thm:weak-Riemann-extension} Let $A\aus \Omega\Subset\CC^n$ be a pure dimensional analytic set. Every holomorphic function on $A^\ast:=\Reg A$ which is bounded at $\Sing A$ is weakly holomorphic on $A$. 
\end{thm}

\bpar
\section{Local \texorpdfstring{$L^2$}{L\textasciicircum 2}-theory of complex curves}\label{sec:local}

In this section, we study the local $L^2$-theory of (locally) irreducible analytic curves in $\CC^{n}$.
By the remarks on the local structure of singular complex curves in \prettyref{sec:resolution} and \prettyref{sec:main}, 
it follows that the studied situation is general enough.
We will compute the $L^2$-Dolbeault cohomology by use of the Puiseux parametrization 
and will see why the term $\sum_{x\in\Sing X} \mult_x' X$ occurs in \eqref{eq:RR-w1}.
\mpar

Let $A$ be an irreducible analytic curve in $\Delta^n\aus\CC^{n}_{zw_1...w_{n-1}}$ given by the Puiseux parametri\-zation
$$\pi\colon\Delta \abb \CC^{n}\ ,\ \pi(t):=(t^s,  w(t)),$$
where $w=(w_1,...,w_{n-1})\colon\Delta\abb \Delta^{n-1}$
is a holomorphic map
such that each component $w_i$ vanishes at least of the order $s+1$ in the origin. Here, $\Delta$ is the unit disk $\{t\in\CC:|t|<1\}$.
We can assume that $\pi$ is bijective, in particular, a resolution/normalization of $A$ such that $\mult_0 A= s$.
Further, we can assume that $0$ is the only singular point of $A$.

For a regular point $(z_0,w_0)\in A^\ast:=\Reg A$, let $t_0\in\Delta^\ast$ be the preimage under $\pi$. Since $\pi$ is biholomorphic on $\Delta^\ast:=\Delta\minus \{0\}$,
$d\pi_{t_0}(\frac{\partial}{\partial t})=st_0^{s-1} \frac{\partial}{\partial z}+\sum_{k=1}^{n-1} w_k'(t_0)\frac{\partial}{\partial w_k}$
is a non-vanishing tangent vector of $A^\ast$ in $(z_0,w_0)$,  \ie 
	$$(1+\|\hbox{$\frac 1s$} t_0^{1-s} {w'(t_0)}\|^2)^{-\spfrac 12}\left(\frac{\partial}{\partial z}+\sum_{k=1}^{n-1} \frac 1s t_0^{1-s}w_k'(t_0)\frac{\partial}{\partial w_k}\right)$$
is a normalized generator of $T_{(z_0,w_0)} A^*$ and $(1+\|\frac 1s t_0^{1-s} {w'(t_0)}\|^2)^{\spfrac 12} dz$ is a normalized generator of $T_{(z_0,w_0)}^\ast A^\ast$.
Since $w'_k$ vanishes at least of order $s$ in the origin, we obtain $1+\|\frac 1s t^{1-s} {w'(t)}\|^2 \sim 1$ on $\Delta$ and $dV_{A^\ast} \sim \ii dz\wedge d\quer z$, where $dV_{A^\ast}$ denotes the volume form on ${A^\ast}$ induced by the standard Euclidean metric of $\CC^n$.
Using $\pi^* dz = d(\pi^\ast z)= dt^s= s t^{s-1} dt$ and $\pi^*(dz\wedge d\overline{z}) = s^2 |t|^{2(s-1)} dt\wedge d\overline{t}$, we get
	\[\pi^\ast dV_{A^\ast}\sim |t|^{2(s-1)} dV_\Delta.\]
Let $\iota:{A^\ast}\abb\Delta^\ast$ be the inverse of $\pi$. 
Then, $\iota(z,w)$ is the root $t=\sqrt[s]{z}$ with $w=w(t)$. 
We get $\iota^\ast(dt)=\frac 1s z^{\spfrac 1s-1} dz$ and $\iota^\ast(dt\wedge d\quer t)=\frac 1{s^2} |z|^{2(\spfrac 1s-1)} dz\wedge d\quer z$, \ie
	\[\iota^\ast dV_{\Delta^\ast}\sim |z|^{2(\spfrac 1s-1)} dV_{A^\ast}.\]
If $g$ is a measurable function on $A^\ast$, we obtain
	\[\int_{A^\ast} |g|^2 dV_{A^\ast} =\int_\Delta |\pi^\ast g|^2\pi^\ast dV_{A^\ast}\sim\int_\Delta|\pi^\ast g|^2\cdot|t|^{2(s-1)} dV_\Delta.\]
Hence,
	\[g\in L^{0,0}({A^\ast})\genau t^{s-1}\pi^\ast g\in L^{0,0}(\Delta).\]
For $(0,1)$-forms and $(1,1)$-forms, we have
	\[\pi^\ast(g d\quer z)=\pi^\ast g\cdot\pi^\ast(d\quer z)=\quer t^{s-1}\pi^\ast (g) d\quer t,\]
	\[\pi^\ast(g dz\wedge d\quer z)=|t|^{2(s-1)}\pi^\ast(g) dt\wedge d\quer t.\]
Thus
	\begin{equation}\label{eq:easypullback}\begin{split}
	f\in L^{0,0}(A^\ast) &\genau t^{s-1}\cdot\,\pi^\ast f\in L^{0,0}(\Delta),\\
	f\in L^{1,0}(A^\ast) &\genau \pi^\ast f\in L^{1,0}(\Delta),\\
	f\in L^{0,1}(A^\ast) &\genau \pi^\ast f\in L^{0,1}(\Delta), \hbox{ and }\\
	f\in L^{1,1}(A^\ast) &\genau t^{1-s}\cdot\,\pi^\ast f\in L^{1,1}(\Delta).
	\end{split}\end{equation}
On the other hand, if $v\in L^{0,0}(\Delta)$, we get
	\[\infty>\int_\Delta |v|^2 dV_\Delta=\int_{A^\ast} |\iota^\ast v|^2\iota^\ast dV_\Delta\sim\int_{A^\ast}|\iota^\ast v|^2\cdot|z|^{2(\spfrac 1s-1)} dV_{A^\ast}.\]
Thus, $|z|^{\spfrac 1s-1}\iota^\ast v$ is square-integrable on $A^*$. For each $(0,1)$-form $v d\quer t\in L^{0,1}(\Delta)$, 
we get $s\iota^\ast(v d\quer t)=\quer z{}^{\spfrac 1s-1}\iota^\ast (v) d\quer z\in L^{0,1}({A^\ast})$, 
and for each $(1,1)$-form $v dt\wedge d\quer t\in L^{1,1}(\Delta)$, we get $|z|^{1-\frac 1s}\iota^\ast(v dt\wedge d\quer t) \in L^{1,1}({A^\ast})$.

\mpar
So, if $f\in L^{0,1}({A^\ast})$, then $u:=\pi^\ast f$ is in $L^2$, too. 
Since $\dim \Delta=1$, there exists $v\in L^{0,0}(\Delta)$ with $\dbar_w v=u$. 
We set $g:=\iota^\ast v$. Since $|z|^{\spfrac 1s-1} g$ is in $L^2$ and $|z|^{2(1-\spfrac 1s)}$ is bounded,
	\[\|g\|^2_{L^2}=\int_{A^\ast}|z^{\spfrac 1s-1} g|^2 \cdot |z|^{2(1-\spfrac 1s)}dV_{A^\ast}\kl\|z^{\spfrac 1s-1} g\|_{L^2}\cdot\| z^{2(1-\spfrac 1s)}\|_{L^\infty}<\infty.\]
Hence, we get an $L^2$-solution for $\dbar_w g=f$ and
	\[H^{0,1}_{w}({A^\ast})= L^{0,1}({A^\ast}) / \cR(\dbar_w)=0.\]

\medskip
In the same way, it is easy to compute 
	\[H^{1,1}_{w}({A^\ast})=0.\]

\medskip
We will now determine $H^{p,0}_{w}({A^\ast})=\ker (\dbar_w: L^{p,0}\abb L^{p,1})$ by use of the $L^2$\nbd extension theorem (\prettyref{thm:extension}). For this, let $\cOL$ be the square-integrable holomorphic functions on $\Delta$, and let $\cOmL$ be the  holomorphic $1$-forms with square-integrable coefficient.
If $g\in L^{0,0}({A^\ast})$ and $\dbar_w g=0$, 
then $v:=\pi^\ast g\in|t|^{1-s}L^{0,0}(\Delta)$ and $\dbar_w v=0$ on $\Delta^\ast$. 
Therefore, $\dbar (t^{s-1}v)=0$ on $\Delta^\ast$ and $t^{s-1}v\in L^{0,0}(\Delta)$. 
The extension theorem implies $\dbar (t^{s-1} v)=0$ on $\Delta$, \ie $v$ is a meromorphic function with a pole of order $s-1$ or less at the origin. 
We say $v\in\cOLm{1-s}$. Since, on the other hand, $\iota^\ast(\cOLm{1-s})\aus\ker\dbar_w$, we conclude
	\begin{equation}\label{eq:puiseux-dbarw-zz}H^{0,0}_{w}({A^\ast})\iso\cOLm{1-s}.\end{equation}
If $f\in L^{1,0}({A^\ast})$ and $\dbar_w f=0$, then $u:=\pi^\ast f\in L^{1,0}(\Delta)$ and $\dbar_w u=0$ on $\Delta$ (using the extension theorem again). 
Hence, $u$ is holomorphic on $\Delta$ and
	\[H^{1,0}_{w}({A^\ast})\iso\cOmL.\]

To compute the cohomology groups $H^{\ast,\ast}_{s}({A^\ast})$, we use $L^2$\nbd duality:

\begin{lem}\label{lem:duality} 
Let $\dbar_e$ denote either the weak or the strong closed extension of $\dbar$, and $\dbar_{e^c}$ the other one.
For $p\in\{0,1\}$, let the range $\cR(\dbar_e)$ of $\dbar_e: L^{p,0}\abb L^{p,1}$ be closed. Then
	\[H^{p,1}_{e}({A^\ast})\iso H^{1-p,0}_{e^c}({A^\ast}).\]
\end{lem}

For the proof see \eg \cite[Thm.~2.3]{Ruppenthal14}.

\begin{lem}\label{lem:rangesp0} For $p\in\{0,1\}$,
	\begin{align*}H^{p,0}_{s}({A^\ast})&\iso H^{1-p,1}_{w}({A^\ast})=0 \hbox{ and}\\
	H^{p,1}_{s}({A^\ast})&\iso H^{1-p,0}_{w}({A^\ast}).\end{align*}
\end{lem}
\begin{proof} Recall that $H_{w}^{1-p,1}({A^\ast})=0$. This implies $L^{1-p,1}(A^\ast)=\cR(\dbar_{w})$ and, particularly,
that the range of $\dbar_{w}: L^{1-p,0}\abb L^{1-p,1}$ is closed.  
As $\thet_{w}=-\qhodge\dbar_{w}\qhodge$ and $\qhodge$ is an isometric isomorphism, we conclude that the range of $\thet_{w}: L^{p,1}\abb L^{p,0}$ is closed as well. This is equivalent to the range of $\dbar_{s}=\thet_{w}^\ast: L^{p,0}\abb L^{p,1}$ being closed (standard functional analysis). \prettyref{lem:duality} implies both isomorphisms.\end{proof}

\mpar
To get the complete picture,
we also need to understand the Dolbeault cohomology groups of the closed extensions $\dbar_{s,w}$ and $\dbar_{w,s}$, respectively.

\begin{lem}\label{lem:sholomorphicA}
For $p\in\{0,1\}$,
	\[H_{w,s}^{p,0}({A^\ast})=0.\]
\end{lem}
\begin{proof} Let $f\in\ker \dbar_{w,s}=H_{w,s}^{p,0}({A^\ast})$. 
We have showed $\omega\cdot u:=\omega\cdot\pi^\ast f\in L^{p,0}(\Delta)$ with $\omega(t)=t^{s-1}$ if $p=0$ and $\omega(t)\equiv1$ if $p=1$. 
By the extension theorem, we conclude $\dbar_s (\omega\cdot u)=0$ on $\Delta$, where $\dbar_s$ denotes the (strong) closure of $\dbar_\cpt:\sC^\infty_{\cpt;p,0}(\Delta)\abb \sC^\infty_{\cpt;p,1}(\Delta)$. 
The generalized Cauchy condition implies that the trivial extension of $\omega u$ to the complex plane
is a holomorphic $p$-form with compact support (\cf \cite[\S\,V.3]{LiebMichel02}). 
We deduce that $\omega u=0$ and, hence, $f=0$.
\end{proof}

\begin{lem}\label{lem:sholomorphicB}
	\begin{align*}H_{s,w}^{0,0}({A^\ast})&\iso\cOL \hbox{ and}\\
	H_{s,w}^{1,0}({A^\ast})&\iso\cOmLm{s-1}.\end{align*}
As $\cOL\iso\widehat\cO_{L^2}(A)$, the first isomorphism implies that
the $\dq_{s,w}$-holomorphic functions on a singular complex curve are precisely the square-integrable weakly holomorphic functions.
\end{lem}

\begin{proof} 
First, we prove that $\cOL=\pi^\ast(\ker \dbar_{s,w}:L^{0,0}(A^\ast)\abb L^{0,1}(A^\ast))$.

\mpar
i) For $v\in\cOL$, we claim that $g:=\iota^\ast v\in\ker\dbar_{s,w}$. To see that,
choose smooth functions $\tilde\chi_k:\RR\abb [0,1]$ with $\tilde\chi_k|_{(-\infty,k]}=0$, 
$\tilde\chi_k|_{[k+1,\infty)}=1$ and $|\tilde\chi_k'|\kl 2$. 
We get $(\tilde\chi_k\nach\log\nach|\log|)'(\rho)=\frac{\tilde\chi'_k(\log|\!\log \rho|)}{\rho\log \rho}$. 
We define $\chi_k:{A^\ast}\abb [0,1], (z,w)\auf \tilde\chi_k(\log|\log|z||)$ (which is inspired by \cite[p.~617]{PardonStern91}) and get $\supp \dbar\chi_k\aus {A^\ast}\durch\Delta^n_{\eps_k}$, where $\eps_k:=\exp(-\exp(k))\abb 0 \konv k \infty$. 
As $v\in L^{0,0}(\Delta)$, we have $g\in z^{1-\frac 1s} L^{0,0}({A^\ast})\aus L^{0,0}({A^\ast})$. 
Then $g\cdot\chi_k\abb g$ in $L^2$. As a holomorphic function, $v$ is bounded in a neighborhood of $0$. 
Therefore,
	\begin{equation*}\begin{split}\|g\dbar \chi_k\|^2_{A^\ast}
	&=\left\|g\cdot \frac{\tilde\chi'_k(\log|\log |z||)}{|z|\log |z|}\dbar |z|\right\|^2_{{A^\ast}\durch\Delta^n_{\eps_k}}
	\lesssim\ \left\|g\cdot\frac{1}{|z|\log |z|}\right\|^2_{{A^\ast}\durch\Delta^n_{\eps_k}}\\
	&\sim\left\|v\cdot\frac{{|t|}^{s-1}}{{|t|}^s\log {|t|}^s}\right\|^2_{\Delta_{\eps_k}} \lesssim\ \left\|\frac{1}{|t|\log |t|}\right\|^2_{\Delta_{\eps_k}}
	=\int_{\Delta_{\eps_k}}\frac{1}{{|t|}^2\log^2|t|}dV\\
	&=2\pi\int_0^{\eps_k}\frac{\rho}{\rho^2\log^2 \rho}d\rho\sim\left[-\frac{1}{\log \rho}\right]_0^{\eps_k}\abb 0,\konv{k}{\infty}.\end{split}
\end{equation*}
Hence, $\dbar(g\chi_k)=g\dbar\chi_k\abb 0=\dbar_w g$ in $L^2$. So, $g\in\dom\dbar_{s,w}$.

\mpar
ii) $\pi^\ast(\ker \dbar_{s,w})\aus\cOL$ (\cf the proof of Lem.~6.2 in \cite{Ruppenthal14}):
Let $g$ be in $\ker\dbar_{s,w}$, \ie there are $g_j$ in $L^2({A^\ast})$ with $g_j\abb g$, $\dbar g_j\abb 0$ in $L^2({A^\ast})$ and $0\notin\supp g_j$. 
Let $\chi\in\sC^\infty_{\cpt}(\Delta,[0,1])$ be identically $1$ on $\Delta_{\spfrac 12}$. 
We define $u:=\chi \pi^\ast g$ and $u_j:=\chi \pi^\ast g_j$.
It follows that $t^{s-1}u_j\abb t^{s-1}u$ and $\dbar u_j\abb\dbar u$ in $L^2(\Delta)$. 
Let $P:L^2(\Delta)\abb L^2(\Delta)$ be the Cauchy-operator on the punctured disc, \ie
	\[[P(h)](t):=\frac 1{2\pi\ii}\int_{\Delta^\ast}\frac{h(\zeta)}{\zeta-t}d\zeta\wedge d\quer\zeta.\]
Since the support of $u_j$ is away from 0 and $\partial\Delta$, we get $u_j=P\left(\frac{\partial u_j}{\partial\quer\zeta}\right)$.
The $L^2$\nbd continuity of $P$ and $\dbar u_j\abb \dbar u$ in $L^2$ imply that
	\[u_j=P\left(\frac{\partial u_j}{\partial\quer\zeta}\right)\abb P\left(\frac{\partial u}{\partial\quer\zeta}\right)\]
in $L^2$. Since $t^{s-1}$ is bounded, we obtain $t^{s-1}u_j\abb t^{s-1}P\left(\frac{\partial u}{\partial\quer\zeta}\right)$ and, hence, $u=P\left(\frac{\partial u}{\partial\quer\zeta}\right)$ in $L^2$.
That yields $\pi^\ast g \in L^2(\Delta)$. With $\pi^\ast g\in\cOLm{1-s}$ and the extension theorem, we conclude $\pi^\ast g\in\cOL$.

\mpar
Second, we claim that $\ker(\dbar_{s,w}:L^{1,0}(A^\ast)\abb L^{1,1}(A^\ast))\iso\cOmLm{s-1}$.
$f=g dz$ is in $\ker\dbar_{s,w}$ iff $g\in\ker\dbar_{s,w}$. This is equivalent to $\pi^\ast g\in\cOL$. Since $\pi^\ast(dz)=t^{s-1}dt$, we infer 
that $\pi^\ast:H_{s,w}^{1,0}({A^\ast})\abb\cOmLm{s-1}, \pi^\ast f=t^{s-1}\pi^\ast gdt$ is an isomorphism.

\mpar
The Riemann extension theorem (\prettyref{thm:weak-Riemann-extension}) implies $\cOL\iso\widehat\cO_{L^2}(A)$ and the last statement.
\end{proof}

\begin{lem}\label{lem:sholomorphicC}
For $p\in\{0,1\}$,
$\cR(\dbar_{w,s}: L^{p,0}({A^\ast})\abb L^{p,1}({A^\ast}))$ and $\cR(\dbar_{s,w}: L^{p,0}({A^\ast})\abb L^{p,1}({A^\ast}))$ both are closed,
and
\[H_{w,s}^{p,1-q}({A^\ast})\iso H_{s,w}^{1-p,q}({A^\ast})\ \mbox{ for } q\in\{0,1\}.\]
\end{lem}

\begin{proof} 
Since $\dbar_{w,s}^\ast=\thet_{s,w}$, $\dbar_{s,w}^\ast=\thet_{w,s}$, $\thet_{s,w}=-\qhodge\dbar_{s,w}\qhodge$ and $\thet_{w,s}=-\qhodge\dbar_{w,s}\qhodge$ 
(see \prettyref{lem:comp-adjoint}), it is easy to see that \prettyref{lem:duality} holds for $\dbar_{w,s}$ and $\dbar_{s,w}$. 
We remark (\cf the proof of \prettyref{lem:rangesp0}) that
	\begin{equation*}\begin{split} \cR(\dbar_{s,w}) \hbox{ is closed } &\genau \cR(\thet_{w,s}) \hbox{ is closed } \\ 	&\genau \cR(\dbar_{w,s}) \hbox{ is closed.}\end{split}\end{equation*}
Therefore, it is enough to show that $\cR(\dbar_{w,s})$ is closed.

\par
Let $\ph:\Delta^\ast\abb\RR$ be the smooth function defined by $\ph(t):=(1-s)\log |t|^2$. Then $L^{p,q}(\Delta,\ph)=t^{1-s} L^{p,q}(\Delta)$
for the $L^2$-space $L^{p,q}(\Delta,\ph)$ with weight $e^{-\ph}$.

We set $T_1:=\pi^\ast\dbar_{w,s}\iota^\ast:  L^{0,0}(\Delta,\ph)\abb L^{0,1}(\Delta)$. 
The extension theorem implies that $T_1$ is the (strong) closure of $\dbar_\cpt: L^{0,0}(\Delta,\ph)\abb L^{0,1}(\Delta)$. Therefore, $T_1^\ast$ is the weak closed extension of $\thet_\cpt^{\ph}: L^{0,1}(\Delta)\abb L^{0,0}(\Delta,\ph)$ which is defined by
	\[(\dbar_\cpt \alpha,\beta)=(\alpha,\thet_\cpt^{\ph} \beta)_\ph=\int\langle\alpha,\thet^\ph_\cpt\beta\rangle e^{-\ph}dV.\]
We set $\qhodge_\ph:=e^{-\ph}\qhodge$. Then $T_2:=-\qhodge_\ph T_1^\ast \qhodge$ is the weak closed extension of $\dbar_\cpt: L^{1,0}(\Delta)\abb L^{1,1}(\Delta,-\ph)$  because integration by parts implies $\thet_\cpt^\ph=-\qhodge_\ph\dbar_\cpt\qhodge$:
	\begin{equation*}\begin{split}(\alpha,\qhodge_{-\ph}\dbar_\cpt\qhodge \beta)_\ph
	&=\int \alpha\wedge \qhodge_{\ph}\qhodge_{-\ph}\dbar_\cpt\qhodge \beta=(-1)^{1-p}\int \alpha\wedge \dbar_\cpt\qhodge \beta\\
	&=-\int \dbar_\cpt \alpha\wedge\qhodge \beta =-(\dbar_\cpt \alpha,\beta).\end{split}\end{equation*}
Hence, $T_2$ is $\dbar_w: L^{1,0}(\Delta)\abb L^{1,1}(\Delta,-\ph)$ in sense of distributions. Since for all $u\in L^{1,1}(\Delta,-\ph)=t^{s-1} L^{1,1}(\Delta)\aus L^{1,1}(\Delta)$ there is a $v\in L^{1,0}(\Delta)$ with $T_2 v=\dbar_w v= u$, the range of $T_2$ is closed. Thus, the range of $T_1^\ast$ and the range of $\dbar_{w,s}=\iota^\ast T_1\pi^\ast: L^{0,0}({A^\ast})\abb L^{0,1}({A^\ast})$ are closed as well.\par
We set $S_1:=\pi^\ast\dbar_{w,s}\iota^\ast: L^{1,0}(\Delta)\abb L^{1,1}(\Delta,-\ph)$ and $S_2:=-\qhodge S_1^\ast\qhodge_{\ph}$. Then $S_2$ is the weak closure of $\dbar_	\cpt: L^{0,0}(\Delta,\ph)\abb L^{0,1}(\Delta)$.
	\begin{equation*}\begin{split} \cR(S_2)&=\{u\in L^{0,1}(\Delta):\exists v\in L^{0,0}(\Delta,\ph)=t^{1-s} L^{0,0}(\Delta) \hbox{ with } S_2 v=u\}\\
	 &\um \{u\in L^{0,1}(\Delta):\exists v\in L^{0,0}(\Delta) \hbox{ with } \dbar_w v=u\}= L^{0,1}(\Delta).\end{split}\end{equation*} 
Therefore, $\cR(S_2)= L^{0,1}(\Delta)$  is closed. This implies the claim.\end{proof}

Summarizing, we computed (with $s=\mult_0 A$):
	\settowidth{\arraycolsep}{.}
	\begin{equation}\begin{array}{rclcl}
	H_w^{0,0}(A^\ast)&\iso & H_s^{1,1}(A^\ast) &\iso &\cOLm{1-s},\\[.8ex] 
	H_w^{1,0}(A^\ast)&\iso &H_s^{0,1}(A^\ast)&\iso &\cOmL,\\[.8ex] 
	H_w^{p,1}(A^\ast)&=&H_s^{1-p,0}(A^\ast)&=&0,\\[.8ex] 
	H_{s,w}^{0,0}(A^\ast)&\iso &H_{w,s}^{1,1}(A^\ast)&\iso &\cOL,\\[.8ex] 
	H_{s,w}^{1,0}(A^\ast)&\iso &H_{w,s}^{0,1}(A^\ast)&\iso &\cOmLm{s-1},\hbox{ and}\\[.8ex] 
	H_{s,w}^{p,1}(A^\ast)&=&H_{w,s}^{1-p,0}(A^\ast)&=&0.
	\end{array}\end{equation}

\bpar
\section{\texorpdfstring{$L^2$}{L\textasciicircum 2}-cohomology of complex curves}\label{sec:main}

We will prove Theorem \ref{thm:RR-w} in this section. As a preparation, we consider the following local situation:
Let $A$ be a locally irreducible analytic set of dimension one in a domain $\Omega\Subset\CC^n_{zw_1...w_{n-1}}$ with $\Sing A=\{0\}$, 
let $dV$ denote the volume form on $A^\ast:=\Reg A$ which is induced by the Euclidean metric
and let $z:A\abb\CC_{z}$ be the projection on the first coordinate. 
Let us mention (\cf \eg Prop.\ in \cite[Sect.~8.1]{Chirka89}):

\begin{thm} 
The set of all tangent vectors at a point of a one-dimensional irreducible analytic set in $\CC^n$ is a complex line.
\end{thm}

Thus, we can assume that $C_0(A)=\CC_{z}\kreuz \{0\}\subset \CC_z \times \CC^{n-1}_{w_1...w_{n-1}}$,
and, therefore, $dV\sim dz\wedge d\quer z$ (by shrinking $\Omega$ if necessary).
\par
Let $\pi:M\abb A$ be a resolution of $A$, $x_0:=\pi^{-1}(0)$. Then $Z=(\pi^\ast(z))$ is the unreduced exceptional divisor of the resolution. 
After shrinking $A$ and $M$ again, we can assume that $M$ is covered by a single chart
$\psi:M\abb \CC$ with $x_0\in M$ and $\psi(x_0)=0$. We set $\zeta:=\pi^\ast(z)$ and get $Z=(\zeta)$. $|Z|=(\psi)$ implies $Z-|Z|=(\frac\zeta\psi)$. 
We obtain
	\[\pi^\ast(dz)=d(\pi^\ast z)=\frac{\partial \zeta}{\partial \psi}d\psi\sim \frac\zeta\psi d\psi.\]
Therefore, $\pi^\ast(dV)\sim\left|\frac\zeta\psi \right|^2 d\psi\wedge d\quer\psi$,
and we conclude (recall the definition of line bundles $L_D$ from Section \ref{sec:resolution}):
\begin{equation}
\label{eq:pullback}\begin{split}f\in L^{p,q}(A^\ast) &\genau \left|\frac\zeta\psi \right|^{1-p-q}\hspace{-7mm}\cdot\,\pi^\ast f\in L^{p,q}(M)\\
	&\genau \pi^\ast f\in L^{p,q}(M,L_{(1-p-q)(Z-|Z|)}),\end{split}
\end{equation}
Nagase stated this equivalence already in Lem.~5.1 of \cite{Nagase90Kyoto}. By use of the extension \prettyref{thm:extension}, we get:
	\begin{equation}\label{eq:dbarw}\begin{split}f\in&\ \dom \big(\dbar_w:L^{p,0}(A^\ast)\abb L^{p,1}(A^\ast)\big)\\
	\genau \pi^\ast f\in&\ \dom \big(\dbar_w:L^{p,0}(M,L_{(1-p)(Z-|Z|)})\abb L^{p,1}(M,L_{p(|Z|-Z)})\big).\end{split}
\end{equation}

\medskip
The essential observation for the proof of Theorem \ref{thm:RR-w} is the following:

\begin{thm}\label{thm:computedCo} 
Let $X$ be a compact complex curve and $L\rightarrow X$ a holomorphic line bundle. Let $\pi:M\abb X$ be a resolution of $X$ with exceptional divisor $Z$, and $D$ a divisor on $M$ such that $\pi^* L \cong L_D$, \ie $\cO(\pi^* L) \cong \cO(D)$. Then
	\begin{align*}H_w^{0,0}(X^\ast, L)&\iso H^0(M,\cO(Z-|Z|+D)),\\
	H_w^{0,1}(X^\ast, L) &\iso H^1(M,\cO(Z-|Z|+D)),\\
	H_w^{1,0}(X^\ast, L) &\iso H^0(M,\Omega^1(D)) \iso H^1(M,\cO(-D)), \hbox { and}\\
	H_w^{1,1}(X^\ast, L) &\iso H^1(M,\Omega^1(D)) \iso H^0(M,\cO(-D)).\end{align*}
\end{thm} 
In \cite[\S 5]{Pardon89}, Pardon proved that $H_{(2),\textrm{sm}}^{0,q}(X^\ast)\iso H^q(M,\cO(Z-|Z|)$, where $H_{(2),\textrm{sm}}^{p,q}(X^\ast)$ denotes the $\dbar$-co\-ho\-mol\-o\-gy with respect to smooth $L^2$-forms. We will use similar arguments here.

\begin{proof} Let $x_0$ be in $\Sing X$, and let $A$ be an open neighborhood of $x_0=0$ in $X$ embedded locally in $\CC^n$. 
We assume that $A=A_1\ver\,...\,\ver A_m$ with at $x_0$ irreducible analytic sets $A_i$. 
We obtain resolutions $\pi_i:=\pi|_{\pi^{-1}(A_i)}:M_i\abb A_i$ of $A_i$. The sets $M_i$ are pairwise disjoint in $M$ and, also, the support of the exceptional divisors $Z_i$ of the resolution $\pi_i$. 
We get $Z|_{\pi^{-1}(A)}=\sum_{i=1}^m Z_i$ and $\left.|Z|\right|_{\pi^{-1}(A)}=\sum_{i=1}^m |Z_i|$. 
Therefore, the consideration in the local case (see \eqref{eq:dbarw}) implies that $\dbar_w:L^{p,0}(X^\ast, L)\abb L^{p,1}(X^\ast, L)$ 
can be identified with
	\[\dbar_w:L^{p,0}(M,L_{(1-p)(Z-|Z|)+D})\abb L^{p,1}(M,L_{p(|Z|-Z)+D}).\]
Hence,
	\[H_w^{0,0}(X^\ast, L)\iso\ker(\dbar_w:L^{0,0}(M,L_{Z-|Z|+D})\abb L^{0,1}(M,L_D))\cong H^0(M,\cO(Z-|Z|+D))\]
and 
\[H_w^{1,0}(X^\ast, L)\iso\ker(\dbar_w:L^{1,0}(M,L_D)\abb L^{1,1}(M,L_{|Z|-Z+D})) \cong H^0(M,\Omega^1(D)).\]

Serre duality (see Theorem 2 in \cite[\S~3.10]{Serre55}) implies
	\[H_w^{1,0}(X^\ast,L)\iso H^0(M,\Omega^1(D))\iso H^1(M,\cO(-D)).\]

To prove the other two isomorphisms, consider the following general situation:
Let $E$ be a divisor on $M$, $L_E$ the associated bundle,
and let $\cL^{p,q}_E$ denote the sheaf on $M$ which is defined by $\cL^{p,q}_E(U):=L_\loc^{p,q}(U,L_E)$ for each open set $U\aus M$. 
Let $E'\kl E$ be another divisor. Consider the $\dbar$-operator in the sense of distributions 
$\dbar_w:\cL^{p,0}_{E}\abb\cL^{p,1}_{E'}$. Let $\cC^{p,0}_{E,E'}$ denote the sheaf defined by 
	\[\cC^{p,0}_{E,E'}(U):=\dom \left(\dbar_w:L_\loc^{p,0}(U,L_E)\abb L_\loc^{p,1}(U,L_{E'})\right).\]
Then $\cC_{E,E'}^{p,0}$ is fine and, in particular, $H^1(M, \cC_{E,E'}^{p,0})=0$.
We get the sequence
\begin{equation}\label{eq:shortseq}
0 \rightarrow \Omega^p(E) \rightarrow  \cC^{p,0}_{E,E'}  \overset{\dbar_w}{\longrightarrow} \cL^{p,1}_{E'} \rightarrow 0
\end{equation}
which is exact by the usual Grothendieck-Dolbeault lemma because there is an embedding $\cL^{p,q}_{E'} \subset \cL^{p,q}_E$
(induced by the natural inclusion $\cO(E')\subset \cO(E)$, see \eqref{eq:L2LocInclusion}).

This induces the long exact sequence of cohomology groups
	\[0 \rightarrow \Gamma(M,\Omega^p(E)) \rightarrow \cC^{p,0}_{E,E'}(M) \overset{\dbar_w}{\longrightarrow} \cL^{p,1}_{E'}(M) \rightarrow H^1(M,\Omega^p(E)) \rightarrow  H^1(M,\cC^{p,0}_{E,E'})=0.\] 
Hence, $\cL^{p,1}_{E'}(M)/\dbar_w\cC^{p,0}_{E,E'}(M)\iso H^1(M,\Omega^p(E))$. We conclude
	\[H_w^{0,1}(X^\ast,L)\iso\cL^{0,1}_{D}(M) / \dbar_w \cC^{0,0}_{Z-|Z|+D,D}(M)\iso H^1(M,\cO(Z-|Z|+D))\]
and, using the Serre duality again,
	\[H_w^{1,1}(X^\ast,L)\iso\cL^{1,1}_{|Z|-Z+D}(M) / \dbar_w \cC^{1,0}_{D,|Z|-Z+D}(M)\iso H^1(M,\Omega^1(D))\iso H^0(M,\cO(-D)).\]
\end{proof}

Theorem \ref{thm:RR-w} follows now as a simple corollary by use of the classical Riemann-Roch theorem
for each connected component of the Riemann surface $M$,
keeping in mind that by definition $g(M)=g(X)$, $\deg L=\deg \pi^\ast L=\deg D$ and $\mult'_x X= \sum_{p\in \pi^{-1}(x)} \deg_p(Z-|Z|)$.

\mpar
To deduce also a Riemann-Roch theorem for the $\dbar_s$-cohomology, 
we can use the following $L^2$-version of Serre duality:

\begin{thm} For each $p\in\{0,1\}$, the range of $\dbar_w:L^{p,0}(X^\ast, L)\abb L^{p,1}(X^\ast, L)$ is closed.
In particular, we get
	\[H_w^{p,q}(X^\ast, L)\iso H_s^{1-p,1-q}(X^\ast, L^{-1}).\]
\end{thm}

\begin{proof} Recall the following well\nbd  known fact. 
If $P:H_1\abb H_2$ is a densely defined closed operator between Hilbert spaces with range $\cR(P)$ of finite codimension,
then the range $\cR(P)$ is closed in $H_2$ (see \eg \cite{HenkinLeiterer84}, Appendix~2.4).

\par
As $M$ is compact, 
\prettyref{thm:computedCo} implies particularly that the range of $\dbar_w$ is finite codimensional and, therefore, closed. 
Since $\dbar_s$ is the adjoint of $-\qhodge \dbar_w\qhodge$, the range of $\dbar_s:L^{1-p,0}(X^\ast, L^{-1})\abb L^{1-p,1}(X^\ast,L^{-1})$ is closed as well. That both ranges are closed implies the $L^2$\nbd duality (\cf \prettyref{lem:duality})
	\[H_w^{p,q}(X^\ast,L)\iso\sH_w^{p,q}(X^\ast,L)\iso\sH_s^{1-p,1-q}(X^\ast,L^{-1})\iso H_s^{1-p,1-q}(X^\ast,L^{-1}),\]
where $\sH_{w/s}^{p,q}(X^\ast,L):=\ker \dbar_{w/s}\durch\ker \dbar_{w/s}^\ast$ denotes the space of $\dbar$-harmonic forms with values in $L$.\end{proof}

Therefore, Theorem \ref{thm:computedCo} yields:
	\begin{equation}\begin{split}\label{eq:computed-s-Co}H_s^{0,0}(X^\ast, L)&\iso H_w^{1,1}(X^\ast,L^{-1}) \iso H^0(M,\cO(D)),\\
	H_s^{0,1}(X^\ast, L) &\iso H_w^{1,0}(X^\ast,L^{-1}) \iso H^1(M,\cO(D)),\\
	H_s^{1,0}(X^\ast, L) &\iso H_w^{0,1}(X^\ast, L^{-1}) \iso H^1(M,\cO(Z-|Z|-D)),\hbox{and} \\
	H_s^{1,1}(X^\ast, L) &\iso H_w^{0,0}(X^\ast,L^{-1}) \iso H^0(M,\cO(Z-|Z|-D)).\end{split}\end{equation}
Haskell computed $H_{\cpt}^{0,q}(X^\ast)\iso H^q(M,\cO_M)$, where $H_{\cpt}^{p,q}(X^\ast)$ denotes the $\dbar$-co\-ho\-mol\-o\-gy
with respect to smooth forms with compact support (see Thm.~3.1 in \cite{Haskell89}).
From \eqref{eq:computed-s-Co}, we obtain the dual version of \prettyref{thm:RR-w}, \ie the Riemann-Roch theorem for the $\dbar_s$-cohomology:

\begin{cor}[$\dbar_s$-Riemann-Roch]\label{thm:RR-s} Let $X$ be a compact complex curve with $m$ irreducible components,
$L\abb X$ be a holomorphic line bundle, and $\pi:M\abb X$ be a resolution of $X$. Then,
	\begin{align*}h_s^{0,0}(X^\ast,L)-h_s^{0,1}(X^\ast,L)&=m-g+\deg L, \hbox{and}\\
	h_s^{1,1}(X^\ast,L)-h_s^{1,0}(X^\ast,L)&=m-g+\deg(Z-|Z|)-\deg L,\end{align*}
where $Z$ is the exceptional divisor of the resolution.\end{cor}

In \cite{BrueningPeyerimhoffSchroeder90}, Br\"uning, Peyerimhoff and Schr\"oder proved that
$h_{s}^{0,0}(X^\ast)-h_{s}^{0,1}(X^\ast)=m-g$  and $h_{w}^{0,0}(X^\ast)-h_{w}^{0,1}(X^\ast)=m-g+\deg Z-|Z|$ by computing the indices of 
the differential operators $\dbar_{s}$ and $\dbar_w$.
Schr\"oder generalized this result for vector bundles in \cite{Schroeder89}.

\bpar
\section{Weakly holomorphic functions}\label{sec:weakly}

In this section, we will prove \prettyref{thm:main2} by studying weakly holomorphic functions
and a localized version of the $\dbar_s$-operator.

\par
Recalling the arguments at the beginning of \prettyref{sec:main}, 
it is easy to see that the results of \prettyref{sec:local} generalize to arbitrary complex curves. In particular, the $\dq_{s,w}$-holomorphic functions on a singular complex curve are precisely the square-integrable weakly holomorphic functions (\cf \prettyref{lem:sholomorphicB}), and the $\dq_{s,w}$-equation is locally solvable in the $L^2$-sense (combine \prettyref{lem:sholomorphicA} and \prettyref{lem:sholomorphicC}).

\par
Let $X$ be a singular complex curve, $\cL_X^{p,q}$ the sheaf of locally square-integrable forms, and let $\dbar_w:\cL_X^{p,0}\abb\cL_X^{p,1}$ be the $\dbar$-operator in the sense of distributions. For each open set $U\aus X$, we define $\dbar_{s,\loc}$ on $L^{p,0}_\loc(U)$ by $f\in\dom \dbar_{s,\loc}$ iff $f\in\dom\dbar_w$ and $f\in\dom \left(\dbar_{s,w}:L^{p,0}(V)\abb L^{p,1}(V)\right)$ for all $V\Subset U$ (for more details, see \cite[Sect.~6]{Ruppenthal14}).
Let $\cF_X^{p,0}$ be the sheaf of germs defined by $\cF_X^{p,0}(U):=\dom \left(\dbar_{s,\loc}:L_\loc^{p,0}(U)\abb L_\loc^{p,1}(U)\right)$, and let $\wcO_X$ denote the sheaf of germs of weakly holomorphic functions on $X$.
Then our considerations above yield an exact sequence
\begin{eqnarray}\label{eq:wh}
0 \rightarrow \widehat\cO_X = \ker\dq_{s,\loc} \hookrightarrow \cF^{0,0}_X \overset{\dbar_{s,\loc}}{\longrightarrow}\cL^{0,1}_X \rightarrow 0.
\end{eqnarray}
The sheaves $\cF_X^{0,0}$ and $\cL^{0,1}_X$ are fine and so \eqref{eq:wh} is a fine resolution of $\widehat\cO_X$.
Let $H^{p,q}_{s,\loc}(X^*)$ denote the $L^{2}_\loc$-Dolbeault cohomology on $X^*$ with respect to the $\dq_{s,\loc}$-operator.
Using $\widehat \cO_X=\pi_\ast\cO_M$, we deduce from \eqref{eq:wh}:
\begin{eqnarray*}
H^0(M,\cO_M)\iso H^0(X,\widehat\cO_X) &\!= H^{0,0}_{s,\loc}(X^*),\\
H^1(M,\cO_M)\iso H^1(X,\widehat\cO_X) &\!\cong H^{0,1}_{s,\loc}(X^*),
\end{eqnarray*}
where $\pi:M\abb X$ is a resolution of $X$. That proves \prettyref{thm:main2}.

\bpar
\section{Applications}\label{sec:application}

There are many applications of the classical Riemann-Roch theorem; we will transfer two of them to our situation
to exemplify how the $L^2$-Riemann-Roch theorem can substitute the classical one on singular spaces.

\subsection{Compact complex curves as covering spaces of \texorpdfstring{$\CC\PP^1$}{CP\textasciicircum 1}}\label{sec:cover}

Let $X$ be a compact irreducible complex curve with $\Sing X=\{x_1,...,x_k\}$, 
let ${(h_i)}_{x_i}\in\cO_{x_i}$ be chosen such that ${(h_i)}_{x_i}\widehat\cO_{x_i}\aus \cO_{x_i}$ 
and let $U_i\aus X$ be a (Stein) neighborhood of $x_i$ with $h_i\cdot \widehat\cO(U_i)\aus \cO (U_i)$
(for the existence of the $h_i$, see \eg Thm.~6 and its Cor.~in \cite[\S\,III.2]{Narasimhan66}). 
Choose an $x_0\in X^\ast$ and a (Stein) neighborhood $U_0$ of $x_0$. We can assume that
$U_0$, ..., $U_k$ are pairwise disjoint.

\medskip
We define a line bundle $L\rightarrow X$ as follows.
Let $U_{k+1}=X^*\minus\{x_0\}$ and choose $f_0\in \cO(U_0)$ such that $f_0$ is vanishing to the order $r:=\ord_{x_0}f_0 \geq 1$ in $x_0$,
which we will determine later, but has no other zeros. 
We also set $f_i:=1/h_i$ for $i=1, ..., k$ and $f_{k+1}= 1$,
and consider the Cartier divisor $\{(U_i, f_i)\}_{i=0, ..., k+1}$ on $X$.
Let $L\rightarrow X$ be the line bundle associated to this divisor.
As the $f_i$ have no zeros for $i>0$, there exists a non-negative integer $\delta$
such that $\deg L=r-\delta$. Now choose $r:=g(X)+\delta+1$. It follows that $\deg L=g(X)+1$.
Give $L$ an arbitrary smooth Hermitian metric.

There is a canonical way to identify holomorphic sections of $L$ with meromorphic functions on $X$.
A holomorphic section $s\in\cO(L)$ is represented by a tuple $\{s_i\}_i$ where $s_j/f_j = s_l/f_l$ on $U_j\cap U_l$.
This gives a meromorphic function $\Psi(s)$ by setting $\Psi(s):=s_j/f_j$ on $U_j$.
Note that $\Psi(s)$ has zeros in the singular points $x_1, ..., x_k$ and may have a pole of order $r$ at $x_0\notin \Sing X$.

\medskip
We can now apply our $L^2$-Riemann-Roch theorem. The $\dbar_s$\nbd Riemann-Roch theorem, Corollary \ref{thm:RR-s}, implies $\dim H_s^{0,0}(X^\ast, L)\gr 1-g(X)+\deg L= 2$. Therefore, there is a section $\tau\in L^{0,0}(X^*,L)$ with $\dbar_s \tau=0$ and $\tau_{k+1}$ is non-constant, where $\tau=\{\tau_i\}_{i=0,..,k+1}$ is written in the trivialization as above. This means that $\tau_i\in L^{0,0}(X^*\cap U_i)$, $\dbar_s\tau_i=0$, and $\tau_j/f_j=\tau_{k+1}$ is non-constant on $U_j\cap U_{k+1}$. \prettyref{thm:main2} implies that $\tau_i\in\widehat\cO(U_i)$, $i=1, ..., k+1$. Now consider $\Psi(\tau)$ as defined above, \ie $\Psi(\tau)=\tau_i/f_i$ on $U_i$. We conclude that $\Psi(\tau)/h_i \in \widehat \cO(U_i)$, thus $\Psi(\tau)\in\cO(U_i)$ for $i=1, ..., k$. Moreover, $\Psi(\tau)$ is non-constant, so it cannot be holomorphic on the whole compact space $X$, thus must have a pole of some order $\leq r$ in $x_0$. Thus:
	\begin{align*}\Psi(\tau)&:X\minus \{x_0\}\abb \CC, \hbox{and}\\
	\widetilde \Psi(\tau)&:X\abb \CC\PP^1, x \mapsto \begin{cases}[\Psi(\tau)(x):1], & x\ungl x_0\\ [1:\frac 1{\Psi(\tau)(x)}], &x\in U_0\end{cases}\end{align*}
are finite, open and, hence, analytic ramified coverings (Covering Lemma, see \cite[Sect.\,VII.2.2]{GrauertRemmertCAS}). In particular, $X\minus \{x_0\}$ is Stein (use \eg Thm.~1 in \cite[\S\,V.1]{GrauertRemmertTSS}).

\mpar
\subsection{Projectivity of compact complex curves}
\label{sec:projectivity}

A line bundle $L\abb X$ on a compact complex space is called very ample if its global holomorphic sections induce a holomorphic embedding into the projective space $\CP^N$, \ie if $s_0, ..., s_N$ is a basis of the space of holomorphic sections  $\Gamma(X,\cO(L))$, then the map
	\begin{eqnarray}\label{eq:Phi} \Phi: X\rightarrow \CC\PP^N\ ,\ x\mapsto [s_0(x):...:s_N(x)] , \end{eqnarray} 
given in local trivializations of the $s_i$, defines a holomorphic embedding of $X$ in $\CC\PP^N$.
If some positive power of the line bundle has this property, then we say that it is ample. 
A compact complex space is called projective if there is an ample (and, hence, a very ample) line bundle on it.  

A classical application of the Riemann-Roch theorem is that any compact Riemann surface is projective,
and a line bundle on a Riemann surface is ample if its degree is positive (\cf \eg \cite[Sect.~10]{Narasimhan92}).
This generalizes to singular complex curves:

\begin{thm}\label{thm:proj} 
Let $X$ be a compact locally irreducible complex curve. 
If $L\abb X$ is a holomorphic line bundle with $\deg L\gg 0$, then $L$ is very ample. 
In particular, $X$ is projective and each holomorphic line bundle on $X$ is ample if its degree is positive.
\end{thm}

Clearly, this result is well-known and follows from more general sheaf-theoretical methods (vanishing theorems)
once one knows that $L$ is positive iff $\deg L >0$
(\cf \eg Thm.~4.4 in \cite[Sect.~V.4.3]{Peternell94} or Satz 2 in \cite[\S\,3]{Grauert62}).
Nevertheless, it seems interesting to us to present another proof of Theorem \ref{thm:proj} which is based on 
the $L^2$-Riemann-Roch of singular complex curves.
The assumption that $X$ must be locally irreducible in \prettyref{thm:proj} is not necessary. 
One can prove the result without this assumption easily by the same technique.
Yet, to keep the notation simple, we present here only the locally irreducible case.

Let us make some preparations for the proof of Theorem \ref{thm:proj}.
Let $X$ be a connected complex curve and $\pi:M\abb X$ a resolution of $X$. We choose a point $x_0\in\Sing X$ and a small neighborhood $U\aus X$ of $x_0$ with $U^\ast:=U\minus\{x_0\}\aus \Reg X$. Assume $X$ is irreducible at $x_0$. We define $p_0:=\pi^{-1}(x_0)$, $V:=\pi^{-1}(U)$, and $V^\ast:=V\minus\{p_0\}$. We can assume that there is a chart $t:V\abb\CC$ such that the image of $t$ is bounded.

\par
The Riemann extension theorem implies that $\pi^{-1}:U\abb V$ is weakly holomorphic or, briefly, $\tau:=t\nach \pi^{-1}\in\wcO(U)$ (see \prettyref{thm:weak-Riemann-extension}).
We show that $\tau$ generates the weakly holomorphic functions at $x_0$ in the following sense:
Let $f\in\wcO(U)$. 
Then $f\nach\pi$ is holomorphic on $V^\ast$ and bounded in $p_0$. This implies that $f\nach\pi$ is holomorphic on $V$,
$f\nach\pi(t)=\sum_{\iota=0}^\infty a_\iota t^\iota$, and $f(x)=\sum a_\iota \tau(x)^\iota$ (by shrinking $U$ and $V$ if necessary). 
This allows to define the order $\ord_{x_0} f$ of vanishing of $f$ in $x_0$  by $r\in\NN_0$ if $a_r\ungl 0$ and $a_\iota=0$ for $\iota<r$. In particular, 
$$\ord_{x_0} f=\ord_{p_0} (f\nach\pi).$$
Note that this definition does not depend on the resolution as different resolutions are biholomorphically equivalent.

\mpar
The $L^2$\nbd extension theorem (see \prettyref{thm:extension}) and \eqref{eq:dbarw} imply
	\begin{equation}\label{eq:application-dbarw-zz}f\in H_w^{0,0}(U) \genau t^{r_0}\cdot \pi^\ast f\in\cO_{L^2}(V)
	\genau(\tau^{r_0}\cdot f\in\wcO(U)\mbox{ and } f\in L^2(U)),
	\end{equation}
where $Z$ denotes the exceptional divisor of the resolution and $r_0:=\deg_{p_0}(Z-|Z|)$. 
In particular, we get the representation $f(x)=\sum_{\iota\gr -r_0} a_\iota \tau(x)^\iota$ and $\ord_{x_0} f:= \ord_{p_0} \pi^*f \gr -r_0$ 
is again well\nbd defined. $f$ is weakly holomorphic iff $\ord_{x_0}f\gr 0$.

\mpar
We denote by $L_{x_0}$ the holomorphic line bundle on $X$ which is trivial on $X\minus\{x_0\}$ and is given by $\tau$ on $U$, \ie the line bundle on $X$ given by the open covering $U_1:=X\minus\{x_0\}, U_0:=U$ and the transition function $g_{01}:=\tau:U_0\durch U_1\abb\CC$. Then 
$\pi^\ast L_{x_0} \cong L_{p_0}$, where $L_{p_0}$ is the holomorphic line bundle $L_{p_0} \rightarrow M$ associated to the divisor $\{p_0\}$.

\par
Let $L\abb X$ be any holomorphic line bundle, $L':=L\tensor L_{x_0}^{-1}$,
and let $s'$ be a section in $H_w^{0,0}(X^\ast, L')$. 
We can assume that $L$ and $L'$ are given by divisors $\{(U_j, f_j)\}$ and $\{(U_j, f'_j)\}$, respectively, where $\{U_j\}$ is an open covering of $X$ with $U_0=U$ and $x_0\notin U_j$ for $j\ungl 0$ and where $f_j, f'_j\in \cM(U_j)$ with $g_{jk}:=f_j/f_k$ and $g'_{jk}:=f'_j/f'_k$ in $\cO(U_{j}\durch U_k)$ ($g_{j,k}$ and $g'_{jk}$ are the transition functions of $L$ and $L'$, respectively). 

We get $f_0=f'_0\cdot \tau$ and $f_j=f'_j$ for $j\ungl 0$. There is a meromorphic function $\tilde s:=\Psi(s')\in\cM(X)$ representing $s'$. 
This meromorphic function is defined by $\tilde s=s'_j/ f'_j$ on $U_j$, where $s'_j$ is the trivialization of $s'$ on $U_j$. We can define a section $s=\{s_j\}\in H_w^{0,0}(X^\ast,L)$ by $s_j=\tilde s \cdot f_j$. Thus $s_0=s'_0\cdot \tau$ and $s_j=s'_j$ for $j\ungl 0$. Hence, $\ord_{x_0} s_0=\ord_{x_0} s'_0+1$.
Summarizing, we get an injective linear map 
$$T:H_w^{0,0}(X^\ast,L\tensor L_{x_0}^{-1})\abb H_w^{0,0}(X^\ast,L),\ s'\mapsto s,$$
which we call the \newterm{natural inclusion}.
It follows from the construction above and by use of \eqref{eq:application-dbarw-zz}
that each section $s\in H_w^{0,0}(X^\ast,L)$ with $\ord_{x_0} s_0>-r_0$ is in the image of $T$.

\mpar
As $H^1(M,\cO(D'))=0$ for a divisor $D'$ with $\deg D'>2g-2$ by the classical Riemann-Roch theorem
(\cf \eg \cite[Sect.~10]{Narasimhan92}), \prettyref{thm:computedCo} (more precisely, $H_w^{0,1}(X^\ast, L)\iso H^1(M,\cO(Z-|Z|+D))$) implies the following vanishing theorem.

\begin{thm}\label{thm:vanishing} If $L\abb X$is a holomorphic line bundle on an irreducible compact complex curve $X$ with $\deg L>2g-2-\sum_{x\in\Sing X} \mult_x' X$, then $H_{w}^{0,1}(X^\ast, L)=0$.\end{thm}

As a preparation for the proof of Theorem \ref{thm:proj}, we get our main ingredient:

\begin{lem}\label{lem:tool}
Let $L\abb X$ be a holomorphic line bundle on a connected compact locally irreducible complex curve $X$ with $\deg L>2g-1-\sum_{x\in\Sing X} \mult_x' X$. 
Then the natural inclusion
$$T:H_w^{0,0}(X^\ast,L\tensor L_{x_0}^{-1})\abb H_w^{0,0}(X^\ast,L)$$ 
is not surjective. 
If $\deg L>2g+r_0-1-\sum_{x\in\Sing X} \mult_x' X$, then there is a section $s\in H_w^{0,0}(X^\ast,L)$ which is weakly holomorphic on $U(x_0)$
and does not vanish in $x_0$.
\end{lem}

Recall that $r_0=\mult_{x_0} X -1 = \deg_{p_0} (Z-|Z|)$.

\begin{proof} 
i) As $\pi^\ast(L\tensor L_{x_0}^{-1})\cong \pi^\ast L\tensor L_{p_0}^{-1}$,
we get $\deg L\tensor L_{x_0}^{-1}=\deg L-1>2g-2 -\deg (Z\!-\!|Z|)$. 
The $\dbar_w$\nbd Riemann-Roch theorem and $h_w^{0,1}(X^\ast,L)=0=h_w^{0,1}(X^\ast,L\tensor L_{x_0}^{-1})$ (using \prettyref{thm:vanishing})
yield
	\begin{equation*}\begin{split}h_w^{0,0}(X^\ast,L\tensor L_{x_0}^{-1})&=1\!-\!g+\deg (Z\!-\!|Z|)+\deg L\tensor L_{x_0}^{-1}\\
	&<1\!-\!g+\deg (Z\!-\!|Z|)+\deg L=h_w^{0,0}(X^\ast,L).\end{split}\end{equation*}
Therefore, the natural inclusion $T$ cannot be surjective.

\mpar
ii) The image of $T^{r_0}:H_w^{0,0}(X^\ast,L\tensor L_{x_0}^{-r_0})\abb H_w^{0,0}(X^\ast,L)$ are the sections $s$ with $\ord_{x_0}s_0\gr0$, where $s_0$ is the trivialization of $s$ over $U(x_0)$, \ie the sections where $s_0$ is weakly holomorphic on $U(x_0)$.
As $H_w^{0,0}(X^\ast,L\tensor L_{x_0}^{-r_0-1})\abb H_w^{0,0}(X^\ast,L\tensor L_{x_0}^{-r_0})$ is not surjective 
(use $\deg L\tensor L_{x_0}^{-r_0}=\deg L-r_0>2g-1-\deg (Z\!-\!|Z|)$ and part (i)), 
there is a section $s'\in H_w^{0,0}(X^\ast,L\tensor L_{x_0}^{-r_0})$ with $\ord_{x_0} s'_0=-r_0$ and $\ord_{x_0}(T^{r_0}(s'))_0=0$.
So, $s:=T^{r_0}(s')$ is the section of $H_w^{0,0}(X^\ast,L)$ we were looking for.
\end{proof}

\medskip
\begin{proofX}{of \prettyref{thm:proj}} 
Let $X$ be a connected compact locally irreducible complex curve with $\Sing X=\{x_1,...,x_k\}$, 
and $L\abb X$ a line bundle with $\deg L\gg 0$. 
Following the classical arguments to show that the map $\Phi$ in \eqref{eq:Phi} is a well-defined
holomorphic embedding (see \eg \cite[V.4, Thm.~4.4]{Peternell94}), we have to prove:
\begin{itemize}
\item[(i)] $\Phi$ is well-defined: For $x\in X$, there exists $s\in \Gamma(X,\cO(L))$  such that $s(x)\neq 0$.
\item[(ii)] $\Phi$ is injective: For $x,y\in X$, $x\neq y$, there exists $s\in \Gamma(X,\cO(L))$ such that $s(x)\neq 0$ and $s(y)=0$.
\item[(iii)] $\Phi$ is an immersion: For $x\in X$, the differential $T_x\Phi$ is injective.
\end{itemize}
Since (obviously) $\Phi$ is closed, (ii) and (iii) imply that $\Phi$ is an embedding (see \eg Sect.~1.2.7 in \cite{GrauertRemmertCAS}).

\smallskip
We will prove the statements (i) and (iii) for singular points $x\in \Sing X$.
The case of regular points is simpler and follows easily with the natural inclusion and \prettyref{lem:tool}.
The statement (ii) can be seen just as (i) by imposing the additional condition that $s(y)=0$ in what
we do to prove the statement (i).

\smallskip
Let $\pi:M\abb X$ be a resolution of singularities.
Set $X^\ast=\Reg X$, $M^\ast=\pi^{-1}(X^\ast)$, $p_j:=\pi^{-1}(x_j)$, and $r_j:=\deg_{p_j} (Z-|Z|)$, 
where $Z$ is the unreduced exceptional divisor of the resolution. 
Fix a $\mu\in\{1,...,k\}$ and choose a neighborhood $U_\mu$ of $x_\mu$ such that 
there exist a resolution of the singularities $\pi: V_\mu\abb U_\mu$ and a chart $t:V_\mu\abb \CC$ with $t\nach\pi^{-1}(x_\mu)=0$, and set $\tau:=t\nach \pi^{-1}$.

\smallskip
For each singularity $x_j$, we can choose a function $h_j\in\cO(U_j)$ such that $h_j\cdot \wcO(U_j)\aus\cO(U_j)$ 
for a neighborhood $U_j$ of $x_j$ small enough (see \cite[\S\,III.2]{Narasimhan66}). The number
	\[\eta_j:=\ord_{x_j} h_j\]
is important for our considerations because of the following fact. 
If $f$ is a function on $U_j$ with $\ord_{x_j} f\gr \eta_j$, then $f/h_j$ is bounded at $x_j$ ($\ord_{x_j}f/h_j\gr 0$); this implies $f/h_j\in\wcO(U_j)$ and, hence, $f\in\cO(U_j)$.
For the maximal ideal in $\cO_{X,x_j}$, we get $\am_{x_j} = \{f \in\cO_{X,x_j}\colon \ord_{x_j} f>0\}$ and $\{f \in\cO_{X,x_j}\colon \ord_{x_j} f\gr 2\eta_j\} \aus \am^2_{x_j}$.

\smallskip
We can choose a weakly holomorphic section $\sigma\in H_w^{0,0}(X^\ast,L)$ 
such that $\sigma$ does not vanish in $x_\mu$ and $\ord_{x_j} \sigma\gr\eta_j$ for $j\ungl \mu$. 
This section $\sigma$ exists as we have the natural inclusion (see the construction above)
	\[H_w^{0,0}\left(X^\ast,L\tensor L_{x_\mu}^{-r_\mu}\tensor\bigotimes_{j\ungl \mu}L_{x_j}^{-\eta_j-r_j}\right)\abb H_w^{0,0}(X^\ast,L),\] 
and $\deg L\gg 0$ implies by Lemma \ref{lem:tool} that the natural inclusion 
	\[H_w^{0,0}\left(X^\ast,L\tensor L_{x_\mu}^{-r_\mu-1}\tensor\bigotimes_{j\ungl \mu}L_{x_j}^{-\eta_j-r_j}\right)\abb 
	H_w^{0,0}\left(X^\ast,L\tensor L_{x_\mu}^{-r_\mu}\tensor\bigotimes_{j\ungl \mu}L_{x_j}^{-\eta_j-r_j}\right)\]
is not surjective. 

\smallskip
Note that $\sigma$ is holomorphic on $X-\{x_\mu\}$ but just weakly holomorphic in $x_\mu$.
We will now modify $\sigma$ so that it becomes holomorphic and non-vanishing in $x_\mu$.
Shrink $U_\mu$ such that $\sigma=\sum_{\iota\gr 0} a_\iota \tau^\iota$ on $U_\mu$ with $a_0\ungl 0$. 
Let $\sigma':=\sigma/a_{0}$ so that $\ord_{x_\mu}(\sigma'-1) \geq 1$, 
\ie $\sigma'-1=\sum_{\iota\gr 1} a'_\iota\tau^\iota$ on $U_\mu$. 
Choose as above a $\tilde \sigma\in H_w^{0,0}(X^\ast,L)$ with $\ord_{x_\mu} \tilde\sigma=1$ and $\ord_{x_j} \tilde\sigma\gr\eta_j$ for $j\ungl \mu$. 
Let $\tilde\sigma=\sum_{\iota\gr 1} \tilde a_\iota \tau^\iota$ close to $x_\mu$ with $\tilde a_1\ungl 0$. 
We define $\sigma'':=\sigma'-\frac{a'_{1}}{\tilde a_{1}}\tilde\sigma$. 
Then, $\ord_{x_\mu} (\sigma''-1)\gr 2$ and $\ord_{x_j} \sigma''\gr\eta_j$ for $j\ungl \mu$. 
We repeat this procedure recursively to get a section $\xi=\{\xi_j\} \in H_w^{0,0}(X^\ast, L)$ with $\ord_{x_\mu} (\xi_\mu-1)\gr\eta_\mu$ 
and $\ord_{x_j} \xi_j\gr\eta_j$ for $j\ungl \mu$. Thus, $\xi$ is a holomorphic section on $X$, non-vanishing in $x_\mu$.
That shows (i) for $x=x_\mu$.

\medskip
We will prove (iii) for $x_\mu$. Let $v\in T_{x_\mu} X =(\am_{x_\mu}/\am_{x_\mu}^2)^\ast$ satisfy $v\ungl 0$, \ie there exists an $f\in\am_{x_\mu}$ with $v(f+\am^2_{x_\mu})\neq 0$. We claim
there exists a $g\in\am_{\Phi({x_\mu})}$ with $g\nach \Phi -f \in \am^2_{x_\mu}$. Then $v(g\nach\Phi+\am^2_{x_\mu})=v(f+\am^2_{x_\mu})\neq 0$, \ie $T_x\Phi (v) \neq 0$.

\smallskip
\textit{Proof of the claim:} Replacing 1 with $f=\sum_{\iota\gr 1} f_\iota \tau^{\iota}$, we can repeat the procedure in (i) to construct a section $\xi=\{\xi_j\}\in H_w^{0,0}(X^\ast, L)$
with $\ord_{x_\mu} (\xi_{\mu}-f) \gr 2\eta_\mu$ and $\ord_{x_j} \xi_j\gr\eta_j$ for $j \neq \mu$. We get $\xi$ is holomorphic, $\xi_\mu\in\am_{x_\mu}$ and $\xi_\mu-f\in\am^2_{x_\mu}$.
Let $\Phi$ be defined by $\Phi(x)=[s_0(x){:}...{:}s_N(x)]$ with holomorphic sections $s_i=\{s_{ij}\}$ (see \eqref{eq:Phi}). Hence, we can choose a vector $(g_0,...,g_N)\in\CC^{N{+}1}$ such that $\xi=\sum_i g_i s_i$. 
Because of (i), there exits an $i_0$ such that $c:=s_{i_0}(x_\mu)\ungl 0$ -- we can assume $i_0=0$. We set $U:=\{x\in U_\mu\colon s_{0\mu}(x) \ungl 0\}$ and identify $\{[t_0{:}...{:}t_N]\colon t_{0}=1\}\aus\CP^N$ with $\CC^N$ such that $\Phi|_U\colon U\abb \CC^N$ is defined by $\Phi(x)=\left(\frac{s_{1\mu}(x)}{s_{0\mu}(x)},...,\frac{s_{N\mu}(x)}{s_{0\mu}(x)}\right)$.
Let $g\colon \CC^N \abb \CC$ denote the holomorphic function $g(t_1,...,t_N):=c\cdot (g_0+\sum_{i{=}1}^N g_i t_i)$, \ie
	\[s_{0\mu}\cdot (g\nach \Phi|_{U})=c\,\sum_{i=0}^N g_is_{i\mu}=c\cdot\xi_\mu\]
on $U$. Since $c\,{=}\,s_{0\mu}(x_\mu)\,{\neq}\,0$ and since $f$ and $\frac{c}{s_{0\mu}}{-}1$ are in $\am_{x_\mu}$, we get $g\in\am_{\Phi(x_\mu)}$ and
	\[g\nach \Phi -f = \frac{c}{s_{0\mu}}\left(\xi_\mu -  f\right) + f\cdot\Big({\frac{c}{s_{0\mu}}}-1\Big) \in\am^2_{x_\mu}.\]
\FormelQed
\end{proofX}
\medskip
For this proof, $L$ has to satisfy
	 \[\deg L>2g+\max\{\eta_j\}+\sum_{j=1}^k(\eta_j{+}r_j)-\deg(Z{-}|Z|)=2g+k+\max\{\eta_j\}+\sum_{j=1}^k\eta_j.\]

\newpage
\providecommand{\bysame}{\leavevmode\hbox to3em{\hrulefill}\thinspace}

\end{document}